\documentclass[12pt]{article}
\usepackage[centertags]{amsmath}
\usepackage{amsfonts}
\usepackage{amssymb}
\usepackage{latexsym}
\usepackage{amsthm}
\usepackage{newlfont}
\usepackage{graphicx}
\usepackage{listings}
\usepackage{booktabs}
\usepackage{abstract}
\newlength{\defbaselineskip}
\newcommand{\setlinespacing}[1]%
           {\setlength{\baselineskip}{#1 \defbaselineskip}}

\newcommand{\N}{{\mathbb{N}}}

\newcommand{\actaqed}{\hfill $\actabox$}
{\medskip\noindent \textit{Proof of #1. }}%
{\actaqed \medskip}

\def \bx{{\mathbf x}}
\def \bk{{\mathbf k}}

\def \bs{{\mathbf s}}

\def\mb{{\bar m}}

\def\Tr{{\mathcal T}}
\def\R{{\mathbb R}}
\def\T{{\mathbb T}}

\def \<{\langle}
\def\>{\rangle}

\def \ff{\varphi}

\def\N{{\mathbb N}}
\def\bN{{\mathbf N}}
\def\Z{{\mathbb Z}}

\def \dim{\operatorname{dim}}

\def\bt{\beta}

\def\bN{\mathbf N}
\def\bW{\mathbf W}
\def\bH{\mathbf H}
\def\bB{\mathbf B}

\newcommand{\be}{\begin{equation}}
\newcommand{\ee}{\end{equation}}

\newtheorem{Theorem}{Theorem}[section]
\newtheorem{Lemma}{Lemma}[section]

\newtheorem{Remark}{Remark}[section]

\newtheorem{Corollary}{Corollary}[section]
\numberwithin{equation}{section}

\begin{document}
\title{{Constructive sparse trigonometric approximation for functions with small mixed smoothness}\thanks{\it Math Subject Classifications.
primary:  41A65; secondary: 42A10, 46B20.}}
\author{  V. Temlyakov \thanks{ University of South Carolina, USA, and Steklov Institute of Mathematics, Russia. Research was supported by NSF grant DMS-1160841 }} \maketitle
\begin{abstract}
 {The paper gives a constructive method, based on greedy algorithms, that provides for the classes of functions with small   mixed smoothness the best possible in the sense of order approximation error for the  $m$-term approximation with respect to the trigonometric system. }
\end{abstract}

\section{Introduction}

The paper is a follow up to the author's recent paper \cite{T150}. The main goal of this paper is to extend the results from \cite{T150} on $m$-term trigonometric approximation in $L_p$ of classes $\bW^r_q$ of functions with bounded in $L_q$ mixed derivative of order $r$    to the case of small smoothness $r$. The most important contribution of this paper, alike the paper \cite{T150}, is that it gives a constructive method, based on greedy algorithms, that provides for the classes $\bW^r_q$ the best possible in the sense of order approximation error $\sigma_m(\bW^r_q)_p$. Theory of sparse approximation with respect to the trigonometric system has a long and interesting history. We give a brief description of this history with emphases put on methods of approximation. We introduce some notation. Denote by $\Tr:=\{e^{ikx}\}_{k\in \Z}$ the univariate trigonometric system and by $\Tr^d := \Tr\times\cdots\times\Tr = \{e^{i(\bk,\bx)}\}_{\bk\in\Z^d}$ the multivariate trigonometric system. Define best $m$-term approximations for a function
$$
\sigma_m(f)_p := \inf_{\{c_j\},\{\bk_j\}}\|f-\sum_{j=1}^m c_je^{i(\bk_j,\bx)}\|_p
$$
and for a class $\bW$ of functions
$$
\sigma_m(\bW)_p :=\sup_{f\in\bW}\sigma_m(f)_p.
$$
We study $m$-term approximation problems for classes of functions with mixed smoothness.    
We begin with the definition of a smoothness class in the case of univariate periodic functions. Let for $r>0$ 
\be\label{6.3}
F_r(x):= 1+2\sum_{k=1}^\infty k^{-r}\cos (kx-r\pi/2) 
\ee
and
\be\label{6.4}
W^r_q := \{f:f=\varphi \ast F_r,\quad \|\varphi\|_q \le 1\}.  
\ee

In the multivariate case for $\bx=(x_1,\dots,x_d)$ denote
$$
F_r(\bx) := \prod_{j=1}^d F_r(x_j)
$$
and
$$
\bW^r_q := \{f:f=\varphi\ast F_r,\quad \|\varphi\|_q \le 1\}.
$$
For $f\in \bW^r_q$ we  denote $f^{(r)} :=\varphi$ where $\varphi$ is such that $f=\varphi\ast F_r$, and define
$$
\|f\|_{\bW^r_q} := \|\ff\|_q.
$$

The first results that showed advantage of $m$-term approximation with respect to the univariate trigonometric system $\Tr$ over the classical approximation by the trigonometric polynomials of order $m$ were obtained by Ismagilov \cite{I} in 1974. His results were improved by Maiorov \cite{M} in 1986 to the relation
$$
\sigma_m(F_2)_\infty \asymp m^{-3/2}.
$$
Note, that best approximation of $F_2$ in the uniform norm by trigonometric polynomials of degree $m$ is of order $m^{-1}$. 
Both Ismagilov \cite{I} and Maiorov \cite{M} used constructive methods, based on number theoretical results. They considered the univariate $m$-term approximation with respect to $\Tr$. An interesting phenomenon specific for the multivariate $m$-term approximation was discovered in \cite{T29} and \cite{Tmon} in 1986. It was established that $\sigma_m(\bW^r_q)_p$ decays faster than the Kolmogorov width $d_m(\bW^r_q)_p$ for $1<q<p\le 2$. The proof of upper bounds for the $\sigma_m(\bW^r_q)_p$, $1<q\le p\le 2$, $r>2(1/q-1/p)$, in \cite{T29} and \cite{Tmon} is constructive. It is based on Theorem \ref{T1.1} (see below). Theorem \ref{T1.1} is often used in   approximation of classes with mixed smoothness. We use it in this paper many times. 

A very interesting and difficult case for the $m$-term approximation is the approximation in $L_p$, $p>2$. Makovoz \cite{Mk} used in 1984 the probabilistic Rosenthal inequality for $m$-term approximation in $L_p$, $2<p<\infty$. Later, in 1987, Belinskii \cite{Be2} used the Rosenthal inequality technique to prove the following lemma.

\begin{Lemma}\label{BeL} Let $2< p<\infty$. For any trigonometric polynomial 
$$
t(\theta_n,x):= \sum_{j=1}^n c_je^{ik_jx},\quad \theta_n:=\{k_j\}_{j=1}^n,
$$
and any $m\le n$ there exists $t(\theta_m,x)$ with $\theta_m \subset \theta_n$ such that
$$
\| t(\theta_n,x)-t(\theta_m,x)\|_p \le C(p)(n/m)^{1/2}\|t(\theta_n,x)\|_2.
$$
\end{Lemma}

Lemma \ref{BeL} and its multivariate versions were used in a number of papers on $m$-term trigonometric approximation in $L_p$, $2<p<\infty$ (see, for instance, \cite{Rom1} and references therein). The use of Lemma \ref{BeL} allowed researchers to obtain the right orders of $\sigma_m(\bW)_p$ for different function classes $\bW$ in $L_p$, $2<p<\infty$. However, this way does not provide a constructive method of approximation. Other nonconstructive method for $m$-term trigonometric approximation, which is more powerful than the above discussed probabilistic method was suggested in \cite{DT1} in 1995. The method in \cite{DT1} is based on a nonconstructive result from finite dimensional geometry due to Gluskin \cite{G}. 

Breakthrough results in constructive $m$-term approximation were obtained by application of general theory of greedy approximation in Banach spaces. It is pointed out in \cite{DKTe} (2002) that the Weak Chebyshev Greedy Algorithm provides a constructive proof of the inequality
$$
\sigma_m(f)_p \le C(p)m^{-1/2}\|f\|_A, \quad p\in [2,\infty).
$$
Here 
$$
\|f\|_A := \sum_\bk |{\hat f}(\bk)|,\quad {\hat f}(\bk):= (2\pi)^{-d}\int_{\T^d} f(\bx)e^{-i(\bk,\bx)}d\bx.
$$
In \cite{T12} (2005) a constructive proof, based on the Weak Chebyshev Greedy Algorithm, was given for the following inequality
$$
\sigma_m(f)_\infty \le Cm^{-1/2} (\log (1+N/m))^{1/2} \|f\|_A,
$$
under assumption that $f$ is a trigonometric polynomial of order $N$. 

The following Theorem \ref{T2.6} is from \cite{T150}. We use it in this paper. 
Let
$$
\Pi(\bN,d) :=\bigl \{(a_1,\dots,a_d)\in \R^d  : |a_j|\le N_j,\
j = 1,\dots,d \bigr\} ,
$$
where $N_j$ are nonnegative integers and $\bN:=(N_1,\dots,N_d)$. We denote
$$
\Tr(\bN,d):=\{t:t = \sum_{\bk\in \Pi(\bN,d)} c_\bk e^{i(\bk,\bx)}\}.
$$
Then 
$$
\dim \Tr(\bN,d) = \prod_{j=1}^d (2N_j  + 1) =: \vartheta(\bN).
$$
For a nonnegative integer $m$ denote $\mb:= \max(m,1)$.

\begin{Theorem}\label{T2.6} There exist constructive greedy-type approximation methods $G^p_m(\cdot)$, which provide $m$-term polynomials with respect to $\Tr^d$ with the following properties: for $2\le p<\infty$
\be\label{3.1p}
\|f-G^p_m(f)\|_p \le C_1(d)(\mb)^{-1/2}p^{1/2}\|f\|_A,\quad \|G^p_m(f)\|_A \le C_2(d)\|f\|_A,
\ee
  and for $p=\infty$, $f\in \Tr(\bN,d)$
\be\label{3.1i}
\|f-G^\infty_m(f)\|_\infty \le C_3(d)(\mb)^{-1/2}(\ln \vartheta(\bN))^{1/2}\|f\|_A,\quad \|G^\infty_m(f)\|_A \le C_4(d)\|f\|_A.
\ee
\end{Theorem}

We now formulate the main results of the paper. The main results of the paper are in Section 3, where we consider $m$-term approximation in $L_p$ with $p\in (2,\infty)$. Here is a typical result from Section 3.
\begin{Theorem}\label{T3.1I} Let $1<q\le 2<p<\infty$ and $\bt p' <r<1/q$. Then we have
$$
\sigma_m(\bW^r_q)_p \asymp m^{-(r-\bt)p/2}(\log m)^{(d-1)(r(p-1)-\bt p)}.
$$
The upper bounds are achieved by a constructive greedy-type algorithm.
\end{Theorem}
Theorem \ref{T3.1I} complements the known result from \cite{T150} for large smoothness: 
Let $1<q\le 2<p<\infty$ and $ r>1/q$. Then we have
$$
\sigma_m(\bW^r_q)_p \asymp m^{-r+\eta}(\log m)^{(d-1)(r-2\eta)}.
$$
The upper bounds are achieved by a constructive greedy-type algorithm.

In Section 3 we also consider the case $r=1/q$ and more general smoothness classes $\bW^{a,b}_q$, which we define momentarily. 
We introduce some more notations.
Let $\mathbf s=(s_1,\dots,s_d )$ be a  vector  whose  coordinates  are
nonnegative integers
$$
\rho(\mathbf s) := \bigl\{ \mathbf k\in\mathbb Z^d:[ 2^{s_j-1}] \le
|k_j| < 2^{s_j},\qquad j=1,\dots,d \bigr\},
$$
$$
Q_n :=   \cup_{\|\mathbf s\|_1\le n}
\rho(\mathbf s) \quad\text{--}\quad\text{a step hyperbolic cross},
$$
$$
\Gamma(N) := \bigl\{ \mathbf k\in\mathbb Z^d :\prod_{j=1}^d
\max\bigl( |k_j|,1\bigr) \le N\bigr\}\quad\text{--}\quad\text{a hyperbolic cross}.
$$
 For $f\in L_1 (\T^d)$
$$
\delta_\bs(f):=\delta_{\mathbf s} (f,\mathbf x) :=\sum_{\mathbf k\in\rho(\mathbf s)}
\hat f(\mathbf k)e^{i(\mathbf k,\mathbf x)}.
$$
Let $G$ be a finite set of points in $\mathbb Z^d$, we denote
$$
\Tr(G) :=\left\{ t : t(\mathbf x) =\sum_{\mathbf k\in G}c_{\mathbf k}
e^{i(\mathbf k,\mathbf x)}\right\} .
$$
For the sake of simplicity we shall write  
$\Tr\bigl(\Gamma(N)\bigr) = \Tr(N)$. 

 Along with classes $\bW^r_q$ defined above it is natural to consider some more general classes. We proceed to the definition of these classes. 

Define for $f\in L_1$
$$
f_l:=\sum_{\|\bs\|_1=l}\delta_\bs(f), \quad l\in \N_0,\quad \N_0:=\N\cup \{0\}.
$$
  Consider the class
$$
\bW^{a,b}_q:=\{f: \|f_l\|_q \le 2^{-al}(\bar l)^{(d-1)b}\}.
$$
Define
$$
\|f\|_{\bW^{a,b}_q} := \sup_l \|f_l\|_q 2^{al}(\bar l)^{-(d-1)b}.
$$
It is well known that the class $\bW^r_q$ is embedded in the class $\bW^{r,0}_q$. 
Classes $\bW^{a,b}_q$ provide control of smoothness at two scales: $a$ controls the power type smoothness and $b$ controls the logarithmic scale smoothness. Similar classes with the power and logarithmic scales of smoothness are studied in the recent book of Triebel \cite{Tr}.

In Section 2 we discuss the case $1<q\le p\le 2$. We use the technique developed in \cite{T29} and \cite{Tmon}. 
The main results of Section 2 are the following two theorems. We use the notation 
$\beta:=\beta(q,p):= 1/q-1/p$ and $\eta:=\eta(q):= 1/q-1/2$.  
\begin{Theorem}\label{Wab} Let $1<q\le p\le 2$. We have
$$
 \sigma_m(\bW^{a,b}_q)_{p}
  \asymp  \left\{\begin{array}{ll} m^{-a+\beta}(\log m)^{(d-1)(a+b-2\beta)}, &   a>2\beta,\\
 m^{-a+\beta}(\log m)^{(d-1)b} , &  \beta< a <2\beta,\\ 
 m^{-\beta}(\log m)^{(d-1)b}(\log\log m)^{1/q}, & a=2\beta.\end{array} \right.
$$
\end{Theorem}
\begin{Theorem}\label{Wr} Let $1<q\le p\le 2$, $r>\beta$. We have
$$
 \sigma_m(\bW^r_q)_p \asymp m^{-r+\beta}(\log m)^{(d-1)(r-2\beta)_+}.
$$
\end{Theorem}
In the case $r>2\bt$ Theorem \ref{Wr} is proved in \cite{T29} and \cite{Tmon} and, as it is pointed out in \cite{Rom1}, in the case $\bt<r\le 2\bt$ the order of $\sigma_m(\bW^r_q)_p$ is obtained in \cite{Bel4}. We present a detailed proof of Theorem \ref{Wr} in Section 2 for completeness (for instance, the author could not find the paper \cite{Bel4}).

We formulate some known results from harmonic analysis and from the hyperbolic cross approximation theory, which will be used in our analysis. 

\begin{Theorem}\label{LP} Let $1<p<\infty$. There exist positive constants $C_1(p,d)$ and $C_2(p,d)$, which may depend on $p$ and $d$, such that for each $f\in L_p$
$$
C_1(p,d)\|f\|_p \le \left\|\left(\sum_\bs |\delta_\bs(f,\bx)|^2\right)^{1/2}\right\|_p \le C_2(p,d)\|f\|_p.
$$
\end{Theorem}
\begin{Corollary}\label{LP1} Let $G$ be a finite set of indices $\bs$ and let the operator $S_G$ map a function $f\in L_p$ to a function
$$
S_G(f) := \sum_{\bs\in G} \delta_\bs(f).
$$
Then for $1<p<\infty$
$$
\|S_G(f)\|_p \le C(p,d) \|f\|_p.
$$
\end{Corollary}

\begin{Corollary}\label{LP2} Let $1<p<\infty$. Denote $p_*:=\max(p,2)$ and $p^*:=\min(p,2)$.
Then for $f\in L_p$ we have
$$
C_3(p,d)\left(\sum_\bs\|\delta_\bs(f)\|_p^{p_*}\right)^{1/p_*} \le \|f\|_p \le C_4(p,d)\left(\sum_\bs\|\delta_\bs(f)\|_p^{p^*}\right)^{1/p^*}.
$$
\end{Corollary}

We now proceed to the problem of estimating $\|f\|_p$ in terms of
the array $\bigl\{ \|\delta_{\bs} (f)\|_q  \bigr\}$.
Here and below $p$ and $q$  are  scalars  such
that $1\le q,p\le \infty$. Let an array
$\varepsilon = \{\varepsilon_{\bs}\}$ be given, where
$\varepsilon_{\bs}\ge 0$, $\bs = (s_1 ,\dots,s_d)$,
and $s_j$ are nonnegative integers, $j = 1,\dots,d$.
We denote by $G(\varepsilon,q)$ and $F(\varepsilon,q)$
the following sets of functions
$(1\le q\le \infty)$:
$$
G(\varepsilon,q) := \bigl\{ f\in L_q  : \bigl\|\delta_{\bs} (f)
\bigr\|_q
\le\varepsilon_{\bs}\qquad\text{ for all }\bs\bigr\} ,
$$
$$
F(\varepsilon,q) := \bigl\{ f\in L_q  : \bigl\|\delta_{\bs} (f)
\bigr\|_q
\ge\varepsilon_{\bs}\qquad\text{ for all }\bs\bigr\}.
$$

The following theorem is from \cite{Tmon}, p.29. For the special case $q=2$  see \cite{T29} and \cite{Tmon}, p.86. 

\begin{Theorem}\label{T1.1} The following relations hold:
\begin{equation}\label{1.1}
\sup_{f\in G(\varepsilon,q)}\|f\|_p \asymp\left(\sum_{\bs}
\varepsilon_{\bs}^p2^{\|\bs\|_1(p/q-1)}\right)^{1/p},
\qquad 1\le q < p < \infty ;
\end{equation}
\begin{equation}\label{1.2}
\inf_{f\in F(\varepsilon,q)}\|f\|_p \asymp\left(\sum_{\bs}
\varepsilon_{\bs}^p  2^{\|\bs\|_1(p/q-1)}\right)^{1/p},
\qquad 1< p < q\le\infty ,
\end{equation}
with constants independent of $\varepsilon$.
\end{Theorem}
We will need a corollary of Theorem \ref{T1.1} (see \cite{Tmon}, Ch.1, Theorem 2.2), which we formulate as a theorem.
\begin{Theorem}\label{A} Let $1<q\le 2$. For any $t\in \Tr(N)$ we have
$$
\|t\|_A:=\sum_\bk |\hat t(\bk)| \le C(q,d) N^{1/q} (\log N)^{(d-1)(1-1/q)}\|t\|_q.
$$
\end{Theorem}

The following Nikol'skii type inequalities are from \cite{Tmon}, Chapter 1, Section 2.
\begin{Theorem}\label{NI} Let $1\le q< p <\infty$. For any $t\in \Tr(N)$ we have
$$
\|t\|_p \le C(q,p,d)N^\bt \|t\|_q,\quad \bt=1/q-1/p.
$$
\end{Theorem}

\section{The case $1<q\le p\le 2$}

{\bf Proof of Theorem \ref{Wab}.} In the case $1<q=p\le 2$ the upper bounds follow from approximation by partial sums $S_{Q_n}(\cdot)$. The corresponding lower bounds follow from the proof of the lower bounds of Theorem 2.1 from \cite{Tmon}, Chapter 4. We now assume that $\bt>0$. 
The case $a> 2\beta$ in Theorem \ref{Wab}, which corresponds to the first line, was proved for classes $\bW^r_q$ in \cite{T29} (see also \cite{Tmon}, Ch.4). In that proof assumption $f\in \bW^r_q$ was used to claim that $\|f_l\|_q\ll 2^{-rl}$, which means $\|f\|_{\bW^{a,0}_q}<\infty$. Thus, that proof gives the required upper bound for the class  $\bW^{a,0}_q$. That same proof gives the corresponding upper bound for the class $\bW^{a,b}_q$ for all $b$. The proofs from \cite{T29} and \cite{Tmon} are constructive.

Consider now the case $\beta <a<2\bt$. The proof of upper bounds in this case uses the ideas from \cite{T29} and \cite{Tmon}. Take an $n\in \N$ and include in approximation 
the
$$
S_{Q_n}(f):=\sum_{\bs:\|\bs\|_1\le n} \delta_\bs(f).
$$
Choose $N$ such that 
$$
2^N \asymp 2^nn^{d-1}
$$
and for $l\in (n,N]$ include in the approximation $m_l$ blocks $\delta_\bs(f)$, $\|\bs\|_1=l$,  with largest $\|\delta_\bs(f)\|_p$. Denote this set of indices $\bs$ by $G_l$. Then by Theorem \ref{T1.1} and the assumption $f\in \bW^{a,b}_q$ we obtain
\be\label{2.1}
\left(\sum_{\bs: \|\bs\|_1=l} \|\delta_\bs(f)\|_p^q 2^{-l\bt q}\right)^{1/q} \ll \|f_l\|_q \le 2^{-al}l^{(d-1)b}.
\ee
We now need the following well known simple lemma (see, for instance, \cite{Tmon}, p.92).
\begin{Lemma}\label{Lqp} Let $a_1\ge a_2\ge \cdots \ge a_M\ge 0$ and $1\le q\le p\le \infty$.
Then for all $m<M$ one has
$$
\left(\sum_{k=m}^M a_k^p\right)^{1/p} \le m^{-\bt} \left(\sum_{k=1}^M a_k^q\right)^{1/q}.
$$
\end{Lemma}
Applying Lemma \ref{Lqp} to $\{\|\delta_\bs(f)\|_p\}$ we obtain 
\be\label{2.2}
\left(\sum_{\bs: \|\bs\|_1=l, \bs\notin G_l} \|\delta_\bs(f)\|_p^p  \right)^{1/p} \ll (m_l+1)^{-\bt}2^{-(a-\bt)l}l^{(d-1)b}.
\ee
Next, using the Corollary \ref{LP2} of the Littlewood-Paley theorem we derive from (\ref{2.2})
\be\label{2.3}
\|\sum_{\bs: \|\bs\|_1=l, \bs\notin G_l} \delta_\bs(f)\|_p\ll \left(\sum_{\bs: \|\bs\|_1=l, \bs\notin G_l} \|\delta_\bs(f)\|_p^p  \right)^{1/p} \ll (m_l+1)^{-\bt}2^{-(a-\bt)l}l^{(d-1)b}.
\ee
Denote
$$
f_l':= \sum_{\bs: \|\bs\|_1=l, \bs\notin G_l} \delta_\bs(f).
$$
Let $\kappa >0$ be such that $a-\bt <\kappa \bt <\bt$. Specify
$$
m_l := [2^{\kappa(N-l)}].
$$
Then (\ref{2.3}) implies
\be\label{2.4}
\|\sum_{n<l\le N} f_l'\|_p\ll \sum_{n<l\le N}(m_l+1)^{-\bt}2^{-(a-\bt)l}l^{(d-1)b} \ll 2^{-(a-\bt)N}N^{(d-1)b}.
\ee
The approximant
$$
A_m(f):= S_{Q_n}(f) + \sum_{n<l\le N} \sum_{\bs: \|\bs\|_1=l, \bs\notin G_l} \delta_\bs(f)
$$
has at most $m$ terms 
\be\label{2.5}
m \ll |Q_n| + \sum_{n<l\le N} 2^l m_l \ll 2^N.
\ee
By Theorem \ref{NI} we have
\be\label{2.6}
\|\sum_{l>N} f_l\|_p \ll \sum_{l>N}\|f_l\|_p \ll \sum_{l>N}\|f_l\|_q 2^{\bt l}\ll 2^{-(a-\bt)N}N^{(d-1)b}.
\ee
Combining (\ref{2.4}) with (\ref{2.6}) and taking into account (\ref{2.5}) we obtain
$$
\|f-A_m(f)\|_p \ll m^{-(a-\bt)} (\log m)^{b(d-1)}.
$$
This completes the proof of upper bounds in the case $\bt<a<2\bt$.

We now proceed to the case $a=2\bt$. We begin with the upper bounds. The proof is as in the above case $\bt <a < 2\bt$. As above, we choose $N$ such that $2^N \asymp 2^nn^{d-1}$. Then $N-n \asymp \log n$. For $l\in (n,N]$ set $m_l=2^{N-l}$. Then as above
\be\label{2.7}
m \ll |Q_n| + \sum_{n<l\le N} 2^l m_l \ll 2^N(N-n).
\ee
In the same way as (\ref{2.6}) was established we get
\be\label{2.8}
\|\sum_{l>N} f_l\|_p \ll \sum_{l>N}\|f_l\|_p \ll \sum_{l>N}\|f_l\|_q 2^{\bt l}\ll 2^{-\bt N}N^{(d-1)b}.
\ee
By (\ref{2.3}) we have
\be\label{2.9}
\|f'_l\|_p \ll (m_l+1)^{-\bt}2^{-\bt l}l^{(d-1)b}
\ee
and
$$
\|\sum_{n<l\le N} f_l'\|_p \ll \left(\sum_{n<l\le N} \|f_l'\|_p^p\right)^{1/p}
$$
\be\label{2.10}
 \ll \left(\sum_{n<l\le N}\left((m_l+1)^{-\bt}2^{-\bt l}l^{(d-1)b} \right)^p\right)^{1/p}\ll 2^{-\bt N}N^{(d-1)b} (N-n)^{1/p}.
\ee
Relations (\ref{2.7}), (\ref{2.8}) and (\ref{2.10}) imply the required upper bound.

We now prove the lower bounds in the case $a=2\bt$. Let $N$ be as above. For $l\in (n,N]$ choose an arbitrary set $B_l$ of $\bs$ such that $\|\bs\|_1=l$ and $|B_l|=m_l:=2^{N-l}$. 
Consider $f$ such that $f_l=0$ for $l\notin (n,N]$ and for $l\in (n,N]$
$$
f_l := 2^{-(2\bt +1 -1/q)l}l^{(d-1)b} m_l^{-1/q} \sum_{\bs\in B_l}\sum_{\bk\in\rho(\bs)} e^{i(\bk,\bx)}.
$$
Then
$$
\|f_l\|_q \ll 2^{-2\bt l} l^{(d-1)b}
$$
and therefore $\|f\|_{\bW^{2\bt,b}_q} \ll 1$. We prove the lower bound for the $\sigma_m(f)_p$ with $m< 2^N(N-n)/8$. Let $K_m:=\{\bk^j\}_{j=1}^m $ be given.  Denote
$$
L:=\{l\in(n,N] : |K_m\cap \cup_{\bs\in B_l}\rho(\bs)|\le 2^N/4\}.
$$
  Then
$$
(N-n-|L|)2^N/4 \le m\le  2^N(N-n)/8,
$$
which implies
$$
|L| \ge (N-n)/2.
$$
Take $l\in L$. Denote
$$
K_m^l := K_m\cap \cup_{\bs\in B_l}\rho(\bs)
$$
and
$$
B_l' := \{\bs\in B_l : |K_m^l \cap \rho(\bs)| \le |\rho(\bs)|/2\}.
$$
As above we derive that
$$
|B_l'| \ge |B_l|/2.
$$
Let $g$ be any polynomial of the form
$$
g=\sum_{\bk\in K_m} c_\bk e^{i(\bk,\bx)}.
$$
By Theorem \ref{T1.1} we get
$$
\|f-g\|_p \gg \left(\sum_{n<l\le N}\sum_{\bs\in B_l} \left(\|\delta_\bs(f-g)\|_22^{l(1/2-1/p)}\right)^p\right)^{1/p}
$$
$$
\gg \left(\sum_{l\in L}\sum_{\bs\in B_l'} \left(\|\delta_\bs(f-g)\|_22^{l(1/2-1/p)}\right)^p\right)^{1/p}
$$
$$
\gg  \left(\sum_{l\in L}\sum_{\bs\in B_l'} \left(2^{-(2\bt+1-1/q)l}l^{(d-1)b} 2^{-(N-l)/q} 2^{l/2} 2^{l(1/2-1/p)}\right)^p\right)^{1/p}
$$
$$
\gg 2^{-\bt N}N^{(d-1)b} |L|^{1/p} \gg 2^{-\bt N} N^{(d-1)b}(N-n)^{1/p}.
$$
Taking into account that $2^N\asymp 2^nn^{d-1}$ and $m\le 2^N(N-n)/8$ we complete the proof of lower bounds.
\newline
{\bf Proof of Theorem \ref{Wr}.} In the case $r\neq 2\bt$ the upper bounds in Theorem \ref{Wr} follow from Theorem \ref{Wab} by the embedding of $\bW^r_q$ into $\bW^{r,0}_q$. It turns out that in the case $r=2\bt$ the above way does not give a sharp upper bound. We now proof the corresponding upper bound in the case $r=2\bt$. We begin with an analog of Lemma \ref{Lqp}.
\begin{Lemma}\label{Lqpw} Let $\{w_j\}_{j=1}^M$ be a set of positive weights. Let $a_1\ge a_2\ge \cdots \ge a_M\ge 0$ and $1\le q\le p\le \infty$.
Then for all $m<M$ one has
$$
\left(\sum_{k=m}^M a_k^pw_k\right)^{1/p} \le \left(\sum_{k=1}^m w_k\right)^{-\bt} \left(\sum_{k=1}^M a_k^qw_k\right)^{1/q}.
$$
\end{Lemma}
\begin{proof} Monotonicity of $\{a_k\}_{k=1}^M$ implies
$$
a_m^q\sum_{k=1}^m w_k \le \sum_{k=1}^M a_k^q w_k,
$$
and 
$$
a_m \le \left(\sum_{k=1}^m w_k\right)^{-1/q} \left(\sum_{k=1}^M a_k^q w_k\right)^{1/q}.
$$
Therefore, 
$$
\left(\sum_{k=m}^M a_k^pw_k\right)^{1/p} \le a_m^{(p-q)/p}\left(\sum_{k=m}^M a_k^q w_k\right)^{1/p} \le
 \left(\sum_{k=1}^m w_k\right)^{-\bt} \left(\sum_{k=1}^M a_k^qw_k\right)^{1/q}.
$$
\end{proof}
  We now prove the upper bound in the case $r=2\bt$. Let $n\in \N$ and, as above, $N$ be such that $2^N\asymp 2^nn^{d-1}$. For $f\in \bW^r_q$ we include in the approximation $S_{Q_n}(f)$ and approximate 
$$
g:=g(f):=\sum_{\bs\in \Delta(n,N)} \delta_\bs(f),\quad \Delta(n,N):=\{\bs: n<\|\bs\|_1\le N\}.
$$
Using Theorem \ref{T1.1} we obtain
\be\label{2.11}
\left(\sum_{\bs\in \Delta(n,N)}\left(2^{\bt \|\bs\|_1}\|\delta_\bs(f)\|_p\right)^q\right)^{1/q}\ll \|f\|_{\bW^r_q}.
\ee
We want to apply Lemma \ref{Lqpw}. Consider $\{v_\bs\}_{\bs\in\Delta(n,N)}$, $v_\bs:=\|\delta_\bs(f)\|_p2^{-\|\bs\|_1/p}$ with weights $w_\bs:= 2^{\|\bs\|_1}$. Then (\ref{2.11}) gives
$$
\left(\sum_{\bs\in \Delta(n,N)}v_\bs^q w_\bs\right)^{1/q} \ll \|f\|_{\bW^r_q}.
$$
Choose $k$ largest $v_\bs$ and denote the corresponding set of indices $\bs$ by $G(k)$. By Lemma \ref{Lqpw} we obtain from the above estimate
\be\label{2.12}
\left(\sum_{\bs\in \Delta(n,N)\setminus G(k)}v_\bs^p w_\bs\right)^{1/p} \ll \left(\sum_{\bs\in G(k)}w_\bs\right)^{-\bt}\|f\|_{\bW^r_q}.
\ee
By the Corollary \ref{LP2} to the Littlewood-Paley Theorem we find
\be\label{2.13}
\|\sum_{\bs\in \Delta(n,N)\setminus G(k)} \delta_\bs(f)\|_p \ll \left(\sum_{\bs\in \Delta(n,N)\setminus G(k)}\|\delta_\bs(f)\|_p^p\right)^{1/p}.
\ee
Combining (\ref{2.13}) with (\ref{2.12}) we obtain
$$
\left(\sum_{\bs\in \Delta(n,N)\setminus G(k)}\|\delta_\bs(f)\|_p^p\right)^{1/p}=\left(\sum_{\bs\in \Delta(n,N)\setminus G(k)}v_\bs^p w_\bs\right)^{1/p}
$$
\be\label{2.14}
  \ll \left(\sum_{\bs\in G(k)}w_\bs\right)^{-\bt}\|f\|_{\bW^r_q}.
\ee
Choose $k$ such that
$$
2^N \le \sum_{\bs\in G(k)} w_\bs <2^{N+1}.
$$
In this way we have constructed an $m$-term approximation of $f$ with $m\ll 2^N$ and error
$$
\|f-S_{Q_n}(f) - \sum_{\bs\in G(k)} \delta_\bs(f)\|_p \ll 2^{-\bt N} \ll m^{-\bt}.
$$
The upper bounds in Theorem \ref{Wr} are proved.  

The lower bounds in the case $r>2\bt$ are proved in \cite{Tmon}. The lower bounds in the case $\bt< r\le 2\bt$ follow from the univariate case.

\section{The case $1<q\le 2<p<\infty$}

The main goal of this section is to prove Theorem \ref{T3.1I} from the Introduction. We reformulate it here for convenience. 
\begin{Theorem}\label{T3.1} Let $1<q\le 2<p<\infty$ and $\bt p' <r<1/q$. Then we have
$$
\sigma_m(\bW^r_q)_p \asymp m^{-(r-\bt)p/2}(\log m)^{(d-1)(r(p-1)-\bt p)}.
$$
The upper bounds are achieved by a constructive greedy-type algorithm.
\end{Theorem}
\begin{proof} We will prove the upper bounds for a wider class $\bW^{r,0}_q$. Let $f\in \bW^{r,0}_q$. Let $n\in \N$. We build an $m$-term approximation with $m=2|Q_n|\asymp 2^nn^{d-1}$. We include in the approximation $S_{Q_n}(f)$. We split the remainder function into two functions
$$
f-S_{Q_n}(f) = g_A+g_0.
$$
We use Theorem \ref{T2.6} to approximate $g_A$ and approximate $g_0$ by $0$. We now describe a construction of $g_A$. First, we choose $N\in (n,Cn]$, $C=C(p,d)$, which will be specified later on, and include in the $g_A$ the
$$
g_A^1 := \sum_{n<l\le N} f_l.
$$
Then by Theorem \ref{A} we have
$$
\|g_A^1\|_A = \sum_{n<l\le N} \|f_l\|_A \ll \sum_{n<l\le N} \|f_l\|_q2^{l/q}l^{(d-1)(1-1/q)}
$$
\be\label{3.1}
\ll \sum_{n<l\le N} 2^{l(1/q-r)}l^{(d-1)(1-1/q)} \ll 2^{N(1/q-r)}N^{(d-1)(1-1/q)}.
\ee
Next, for $l>N$ define
$$
u_l:=[n^{d-1}2^{\kappa (N-l)}]
$$
with $\kappa$ satisfying
$$
\frac{1/q-r}{1-1/q} <\kappa <\frac{r-\bt}{\bt}.
$$
Such $\kappa$ exists because our assumption $r>\bt p'$ is equivalent to the inequality
$$
\frac{1/q-r}{1-1/q}  <\frac{r-\bt}{\bt}.
$$
Denote $G(l)$ the set of indices $\bs$, $\|\bs\|_1=l$, of cardinality $|G(l)|=u_l$, with largest $\|\delta_\bs(f_l)\|_2$. Second, we include in $g_A$ the
$$
g_A^2:=\sum_{l>N}\sum_{\bs\in G(l)} \delta_\bs(f_l).
$$
It is clear that there is only finite number of nonzero terms in the above sum. We have
$$
\|g_A^2\|_A = \sum_{l>N}\|\sum_{\bs\in G(l)} \delta_\bs(f_l)\|_A \ll \sum_{l>N} \|f_l\|_q 2^{l/q}|G(l)|^{1-1/q}
$$
$$
\ll \sum_{l>N} 2^{l(1/q-r)}n^{(d-1)(1-1/q)}2^{\kappa(N-l)(1-1/q)} 
$$
$$
\ll n^{(d-1)(1-1/q)}2^{\kappa N(1-1/q)}\sum_{l>N} 2^{l(1/q-r-\kappa(1-1/q))}.
$$
By our choice of $\kappa$ we have $1/q-r-\kappa(1-1/q)<0$ and, therefore, we continue
\be\label{3.2}
\ll N^{(d-1)(1-1/q)}2^{N(1/q-r)}.
\ee
By Theorem \ref{T2.6} we obtain
\be\label{3.3}
\|g_A-G^p_{m/2}(g_A)\|_p \ll m^{-1/2} N^{(d-1)(1-1/q)}2^{N(1/q-r)}.
\ee

We now bound the $\|g_0\|_p$. Denote
$$
f_l^o := f_l - \sum_{\bs\in G(l)} \delta_\bs(f_l).
$$
By Theorem \ref{T1.1} we get
$$
\left(\sum_{\|\bs\|_1=l}\left(\|\delta_\bs(f_l)\|_22^{\|\bs\|_1(1/2-1/q)}\right)^q\right)^{1/q} \ll 2^{-rl}
$$
and
$$
\left(\sum_{\|\bs\|_1=l}\|\delta_\bs(f_l)\|_2^q\right)^{1/q} \ll 2^{l(-r+1/q-1/2)}.
$$
By Lemma \ref{Lqp} we obtain
$$
\left(\sum_{\bs\notin G(l)}\|\delta_\bs(f_l)\|_2^p\right)^{1/p} \ll (u_l+1)^{-\bt}2^{l(-r+1/q-1/2)}.
$$
By Theorem \ref{T1.1} we get from here
$$
\|f_l^o\|_p\ll \left(\sum_{\bs\notin G(l)}\left(\|\delta_\bs(f_l)\|_22^{\|\bs\|_1(1/2-1/p)}\right)^p\right)^{1/p} \ll 2^{-(r-\bt)l} (u_l+1)^{-\bt}.
$$
Thus
$$
\|g_0\|_p \le \sum_{l>N}\|f_l^o\|_p \ll \sum_{l>N} 2^{-(r-\bt)l}n^{-\bt(d-1)}2^{-\bt\kappa(N-l)}
$$
\be\label{3.4}
\ll n^{-\bt(d-1)}2^{-\bt\kappa N} \sum_{l>N} 2^{-(r-\bt -\bt\kappa)l}.
\ee
By our choice of $\kappa$ we have $r-\bt -\bt\kappa>0$. Therefore, (\ref{3.4}) gives
\be\label{3.5}
\|g_0\|_p \ll 2^{-(r-\bt)N}n^{-\bt(d-1)}.
\ee
We now choose $N$ from the condition
$$
2^{(1/q-r)N}n^{(d-1)(1-1/q)} m^{-1/2} \asymp 2^{-(r-\bt)N}n^{-\bt(d-1)}.
$$
This is equivalent to 
$$
2^N \asymp 2^{np/2}n^{(d-1)(1-p/2)}
$$
or in terms of $m$
\be\label{3.8}
2^N \asymp m^{p/2}(\log m)^{(d-1)(1-p)}.
\ee
As a result it gives us the following upper bound for the error of approximation
$$
\|f-S_{Q_n}(f)-G_{m/2}^p(g_A)\|_p \ll 2^{-(r-\bt)N}n^{-\bt(d-1)}
$$
$$
 \ll m^{-(r-\bt)p/2}(\log m)^{(d-1)(r(p-1)-\bt p)}.
$$
This completes the proof of upper bounds. 

We proceed to the lower bounds. For a given $m$ chose $N$ as in (\ref{3.8}). Consider the function
$$
g(\bx):= \sum_{\|\bs\|_1=N}\sum_{\bk\in \rho(\bs)} e^{i(\bk,\bx)} = \sum_{\bk\in \Delta Q_N} e^{i(\bk,\bx)},\quad \Delta Q_N:=Q_N\setminus Q_{N-1}.
$$
It is known that 
\be\label{3.9}
\|g\|_q\asymp 2^{N(1-1/q)}N^{(d-1)/q},\quad 1<q<\infty.
\ee
We now estimate the $\sigma_m(g)_p$ from below. Take any set $K_m$ of $m$ frequencies $\bk$. Consider an additional function
$$
h(\bx):= \sum_{\bk\in \Delta Q_N \setminus K_m} e^{i(\bk,\bx)}. 
$$
For any polynomial $t$ with frequencies from $K_m$ we have
\be\label{3.10}
\<g-t,h\> \le \|g-t\|_p \|h\|_{p'}
\ee
and
\be\label{3.11}
\<g-t,h\> = \<g,h\> = \sum_{\bk\in \Delta Q_N \setminus K_m} 1 = |\Delta Q_N \setminus K_m|.
\ee
From our choice (\ref{3.8}) of $N$ it is clear that asymptotically 
$$
|\Delta Q_N \setminus K_m| \ge |\Delta Q_N| - m \gg 2^N N^{d-1}.
$$
Next, we have
$$
\|h\|_{p'} \le \|g\|_{p'} + \|g-h\|_{p'} \le \|g\|_{p'} +\|g-h\|_2 
$$
$$
\ll 2^{N/p}N^{(d-1)/p'} +m^{1/2} \ll m^{1/2}.
$$
Thus, (\ref{3.10}) and (\ref{3.11}) yield 
$$
\sigma_m(g)_p \gg 2^N N^{d-1} m^{-1/2}.
$$
We have from (\ref{3.9})
$$
\|g^{(r)}\|_q \ll 2^{N(r+1-1/q)}N^{(d-1)/q}.
$$
Therefore,
$$
\sigma_m(\bW^r_q)_p \gg 2^{N(1/q-r)}N^{(d-1)(1-1/q)} m^{-1/2} 
$$
$$
\asymp   m^{-(r-\bt)p/2}(\log m)^{(d-1)(r(p-1)-\bt p)}.
$$
This proves the lower bounds.

\end{proof}

The above proof of Theorem \ref{T3.1} gives the right order of $\sigma_m(\bW^{a,b}_q)_p$ for $\bt p' <a<1/q$ and all $b$. We formulate this as a theorem.

\begin{Theorem}\label{T3.3} Let $1<q\le 2<p<\infty$ and $\bt p' <a<1/q$. Then we have
$$
\sigma_m(\bW^{a,b}_q)_p \asymp m^{-(a-\bt)p/2}(\log m)^{(d-1)(b+a(p-1)-\bt p)}.
$$
The upper bounds are achieved by a constructive greedy-type algorithm.
\end{Theorem}

For $f\in \bW^{r,0}_q$ one has for $r>\bt$
\be\label{3.12}
\|\sum_{l>N} f_l\|_p \le \sum_{l>N}\|f_l\|_p \ll \sum_{l>N} \|f_l\|_q2^{\bt l} \ll 2^{-(r-\bt)N}.
\ee
In the proof of Theorem \ref{T3.1} we constructed $g_A^2$ and $g_0$. It resulted in a better error estimate of the $m$-term approximation of the tail $\sum_{l>N} f_l$ than the simple bound (\ref{3.12}). We got the error $\ll 2^{-(r-\bt)N}n^{-\bt(d-1)}$. We obtain the same 
improvement of the error if, in addition to the assumption $f\in \bW^{r,0}_q$, we assume that 
$f\in\bW^{r-\bt,-\bt}_p$. For $f\in\bW^{r-\bt,-\bt}_p$ we have
$$
\|\sum_{l>N} f_l\|_p  \ll 2^{-(r-\bt)N}N^{-\bt(d-1)}.
$$
We formulate a theorem, which follows from the proof of Theorem \ref{T3.1}.

\begin{Theorem}\label{T3.4} Let $1<q\le 2<p<\infty$ and $\bt  <a<1/q$. Then we have
$$
\sigma_m(\bW^{a,b}_q\cap \bW^{a-\bt,b-\bt}_p)_p \asymp m^{-(a-\bt)p/2}(\log m)^{(d-1)(b+a(p-1)-\bt p)}.
$$
The upper bounds are achieved by a constructive greedy-type algorithm.
\end{Theorem}

We note that the class $\bH^r_q$ (see the definition in Section 5), $1<q\le 2$, is embedded into the class $\bW^{r,b}_q\cap \bW^{r-\bt,b-\bt}_p$ with $b=1/q$. This follows from Corollary \ref{LP2} and Theorem \ref{T1.1}. The following theorem holds.

 \begin{Theorem}\label{T3.5} Let $1<q\le 2<p<\infty$ and $\bt  <a<1/q$. Then we have
$$
\sigma_m(\bH^r_q)_p \asymp m^{-(r-\bt)p/2}(\log m)^{(d-1)(1/q+r(p-1)-\bt p)}.
$$
The upper bounds are achieved by a constructive greedy-type algorithm.
\end{Theorem}

The order of $\sigma_m(\bH^r_q)_p$ is known (see Romanyuk \cite{Rom1}). However, the corresponding upper bounds in \cite{Rom1} are proved by a  nonconstructive method of approximation. 

We now proceed to the case $a=1/q$. 

 \begin{Theorem}\label{T3.6} Let $1<q\le 2<p<\infty$ and $a=1/q$. Then we have
$$
\sigma_m(\bW^{1/q,b}_q)_p \asymp m^{-1/2}(\log m)^{(d-1)(b+1-1/q)+1}.
$$
The upper bounds are achieved by a constructive greedy-type algorithm.
\end{Theorem}
\begin{proof} The proof goes along the lines of the proof of Theorem \ref{T3.1}. We use the same notation as above. We begin with the upper bounds. In the case $a=1/q$ the bound (\ref{3.1}) reads
\be\label{3.13}
\|g_A^1\|_A \ll (N-n)N^{(d-1)(b+1-1/q)}.
\ee
We repeat the argument from the proof of Theorem \ref{T3.1} for $g_A^2$ and $g_0$ with 
$\kappa \in (0,(a-\bt)/\bt)$. It gives
\be\label{3.14}
\|g_A^2\|_A \ll N^{(d-1)(b+1-1/q)},
\ee
\be\label{3.15}
\|g_0\|_p \ll 2^{-(a-\bt)N}n^{(b-\bt)(d-1)}.
\ee
Choosing $N$ from (\ref{3.8}) we obtain
$$
\|f-S_{Q_n}(f)-G_{m/2}^p(g_A)\|_p \ll (N-n)N^{(d-1)(b+1-1/q)}m^{-1/2}
$$
$$
\ll m^{-1/2}(\log m)^{(d-1)(b+1-1/q)+1}.
$$
This completes the proof of upper bounds. 

We proceed to the lower bounds. For a given $m$ chose $N$ as in (\ref{3.8}). Consider the function
$$
g(\bx):= \sum_{n<l\le N} 2^{-l/q}l^{b(d-1)}\left(2^{l(1-1/q)}l^{(d-1)/q}\right)^{-1} \sum_{\bk\in \Delta Q_l} e^{i(\bk,\bx)},\quad \Delta Q_l:=Q_l\setminus Q_{l-1}.
$$
Then $\|g\|_{\bW^{1/q,b}_q} \ll 1$.

We now estimate the $\sigma_m(g)_p$ from below. Take any set $K_m$ of $m$ frequencies $\bk$. Consider an additional function
$$
h(\bx):= \sum_{\bk\in \Delta (n,N) \setminus K_m} e^{i(\bk,\bx)}. 
$$
For any polynomial $t$ with frequencies from $K_m$ we have
\be\label{3.16}
\<g-t,h\> \le \|g-t\|_p \|h\|_{p'}
\ee
and
\be\label{3.17}
\<g-t,h\> = \<g,h\> = \sum_{\bk\in \Delta (n,N) \setminus K_m} {\hat g}(\bk).  
\ee
From our choice (\ref{3.8}) of $N$ it is clear that asymptotically 
$$
\sum_{\bk\in \Delta (n,N) \setminus K_m} {\hat g}(\bk) \gg (N-n) N^{(d-1)(b+1-1/q)}.
$$
Next, we have
$$
\|h\|_{p'} \le \|g\|_{p'} + \|g-h\|_{p'} \le \|g\|_{p'} +\|g-h\|_2 
$$
$$
\ll 2^{N/p}N^{(d-1)/p'} +m^{1/2} \ll m^{1/2}.
$$
Thus, (\ref{3.16}) and (\ref{3.17}) yield 
$$
\sigma_m(g)_p \gg  (N-n) N^{(d-1)(b+1-1/q)} m^{-1/2}.
$$
Therefore,
$$
\sigma_m(\bW^{1/q,b}_q)_p \gg (N-n)N^{(d-1)(b+1-1/q)} m^{-1/2} 
$$
$$
\asymp  m^{-1/2}(\log m)^{(d-1)(b+1-1/q)+1}.
$$
This proves the lower bounds.

\end{proof}

The above proof of Theorem \ref{T3.6} can be adjusted to prove the following results for the $\bW^{1/q}_q$ classes. 

 \begin{Theorem}\label{T3.7} Let $1<q\le 2<p<\infty$. Then we have
$$
\sigma_m(\bW^{1/q}_q)_p \asymp m^{-1/2}(\log m)^{d(1-1/q)}.
$$
The upper bounds are achieved by a constructive greedy-type algorithm.
\end{Theorem}
\begin{proof} We need the following analog of Theorem \ref{A}. 
\begin{Lemma}\label{L3.1} For $t\in \Tr(\Delta (n,N))$ we have for $1<q\le 2$
$$
\|t\|_A \ll N^{(d-1)(1-1/q)}(N-n)^{1-1/q}\|t\|_{\bW^{1/q}_q}.
$$
\end{Lemma}
\begin{proof} Let $r=1/q$ and
$$
t=\ff \ast F_r,\quad \|\ff\|_q=\|t\|_{\bW^r_q}.
$$
Then
\be\label{3.18}
\|t\|_A\le \|\ff\|_q\left\|\sum_{\bk\in \Delta (n,N)} \epsilon_\bk {\hat F}_r(\bk)e^{i(\bk,\bx)}\right\|_{q'},\quad |\epsilon_\bk|=1.
\ee
Using Theorem \ref{T1.1} with parameters $q'$ and $2$ we obtain
$$
\left\|\sum_{\bk\in \Delta (n,N)} \epsilon_\bk {\hat F}_r(\bk)e^{i(\bk,\bx)}\right\|_{q'}
$$
\be\label{3.19}
  \ll 
\left(\sum_{n<l\le N}\left(2^{-lr} 2^{l/2} 2^{l(1/2-1/q')}\right)^{q'}l^{d-1}\right)^{1/q'} \ll (N-n)^{1/q'}N^{(d-1)/q'}.
\ee
Combining (\ref{3.18}) and (\ref{3.19}) we complete the proof of Lemma \ref{L3.1}. 

\end{proof}
We return to the proof of Theorem \ref{T3.7}. The proof goes along the lines of the proof of Theorem \ref{T3.6}. We use the same notation as above. We begin with the upper bounds. In our case Lemma \ref{L3.1} implies the following analog of the bound (\ref{3.13}) 
\be\label{3.20}
\|g_A^1\|_A \ll (N-n)^{1-1/q}N^{(d-1)(1-1/q)}.
\ee
Choosing $N$ from (\ref{3.8}) we obtain
$$
\|f-S_{Q_n}(f)-G_{m/2}^p(g_A)\|_p \ll (N-n)^{1-1/q}N^{(d-1)(1-1/q)}m^{-1/2}
$$
$$
\ll m^{-1/2}(\log m)^{(d-1)(1-1/q)+1-1/q}=m^{-1/2}(\log m)^{d(1-1/q)}.
$$
This completes the proof of upper bounds. 

The lower bounds follow from the same example (with $b=0$) that was used in the proof of Theorem \ref{T3.6}. In this case instead of $\|g\|_{\bW^{1/q,0}_q} \ll 1$ we have  
$$
 \|g\|_{\bW^{1/q}_q} \ll (N-n)^{1/q}
 $$
which brings the bound
$$
\sigma_m(\bW^{1/q}_q)_p \gg (N-n)^{1-1/q}N^{(d-1)(1-1/q)} m^{-1/2} 
$$
$$
\asymp  m^{-1/2}(\log m)^{d(1-1/q)}.
$$
This proves the lower bounds.

\end{proof}

\section{The case $q=1$}

We begin with the case $2\le p<\infty$. 
\begin{Theorem}\label{T4.1} For any $\epsilon>0$ we have for $2\le p<\infty$
$$
 \sigma_m(\bW^{a,b}_1)_{p}
  \ll  \left\{\begin{array}{ll} m^{-a+1/2}(\log m)^{(d-1)(a-1 +b)+\epsilon}, &    a>1,\\
 m^{-(a-\beta)p/2}(\log m)^{(d-1)b +\epsilon}, &  \bt< a<1,\quad \bt=1-1/p,\\ 
 m^{-1/2}(\log m)^{(d-1)b+1 +\epsilon}, &   a=1,\end{array} \right.
$$
with constants in $\ll$ allowed to depend on $\epsilon$, $d$, and $p$. 

The upper bounds are achieved by a constructive greedy-type algorithm.
\end{Theorem}
\begin{proof} For large smoothness $a>1$ 
the following  lemma from \cite{T150} plays the key role in the proof. 
\begin{Lemma}\label{L4.1}  Define for $f\in L_1$
$$
f_l:=\sum_{\|\bs\|_1=l}\delta_\bs(f), \quad l\in \N_0,\quad \N_0:=\N\cup \{0\}.
$$
Consider the class
$$
\bW^{a,b}_A:=\{f: \|f_l\|_A \le 2^{-al}({\bar l})^{(d-1)b}\}.
$$
Then for $2\le p<\infty$ and $0<\mu<a$ there is a constructive method $A_m(\cdot,p,\mu)$ based on greedy algorithms, which provides the bound for $f\in \bW^{a,b}_A$
\begin{equation}\label{4.0}
\|f-A_m(f,p,\mu)\|_p \ll  m^{-a-1/2} (\log m)^{(d-1)(a+b)},\quad 2\le p<\infty,      
\end{equation}
\end{Lemma}
We also need the following version of Theorem \ref{A} in the case $q=1$ (see \cite{Tmon}, Chapter 1, Section 2). 
\begin{Lemma}\label{L4.2} For any $\epsilon>0$ there is $C(\epsilon,d)$ such that for each $t\in \Tr(N)$ we have
$$
\|t\|_A \le C(\epsilon,d)N(\log N)^\epsilon \|t\|_1.
$$
\end{Lemma}

Let $f\in \bW^{a,b}_1$, $a>1$. By Lemma \ref{L4.2} we get
$$
\|f_l\|_A \ll 2^{-l(a-1)} l^{b(d-1) +\epsilon}
$$
with a constant in $\ll $ allowed to depend on $\epsilon$ and $d$. Setting $b(\epsilon) := b+\epsilon/(d-1)$, we obtain
$$
\|f\|_{\bW^{a-1,b(\epsilon)}_A }\ll 1.
$$
Lemma \ref{L4.1} gives a constructive proof of
$$
\sigma_m(f)_p \ll m^{-a+1/2} (\log m)^{(d-1)(a-1+b)+\epsilon}.
$$
This proves  the first inequality in Theorem \ref{T4.1}.

Consider now the case $\bt<a<1$. The argument in this case is close to the proof of Theorem \ref{T3.1}. We use the same notations. We now define
$$
g_A:= \sum_{n<l\le N} f_l,\qquad g_0:=\sum_{l>N}f_l.
$$
By Lemma \ref{L4.2} we get
\be\label{4.2}
\|g_A\|_A \ll \sum_{n<l\le N} 2^{l(1-a)}l^{b(d-1)+\epsilon} \ll 2^{N(1-a)}N^{b(d-1)+\epsilon}.
\ee
By Theorem \ref{T2.6} we obtain
\be\label{4.3}
\sigma_m(g_A)_p \ll m^{-1/2} 2^{N(1-a)}N^{b(d-1)+\epsilon}.
\ee

By Theorem \ref{NI} 
\be\label{4.4} 
\|g_0\|_p \le \sum_{l>N} \|f_l\|_p \le \sum_{l>N}\|f_l\|_1 2^{\bt l} \ll 2^{-(a-\bt)N}N^{b(d-1)}.
\ee
Close $N$ such that
$$
m^{-1/2}2^{N(1-a)} \asymp 2^{-(a-\bt)N},
$$
that is
\be\label{4.5}
2^N\asymp m^{p/2}.
\ee
This gives the error bound
$$
\sigma_m(f)_p \ll m^{-(a-\bt)p/2}(\log m)^{b(d-1)+\epsilon}.
$$
This proves the required bound in the second case.

In the case $a=1$ we get as in (\ref{4.2})
$$
\|g_A\|_A \ll N^{b(d-1)+1+\epsilon}.
$$
Choosing $N$ from (\ref{4.5}) we obtain
$$
\sigma_m(f)_p \ll m^{-1/2}(\log m)^{b(d-1)+1+\epsilon}.
$$
This completes the proof of Theorem \ref{T4.1}.

We note that in the case $a\le 1$  the corresponding lower bounds with $\epsilon =0$ follow from the univariate case (see \cite{Be2}).  We now prove the lower bounds for $a>1$. It is sufficient to prove them for $p=2$. Let $m$ be given and $n$ be such that $2^n n^{d-1}\asymp m$ and $m\le  c(d)|\Delta Q_n|$ with small enough $c(d)>0$. Let 
 $$
{\mathcal K}_{N} (x) :=
\sum_{|k|\le N} \bigl(1 - |k|/N\bigr) e^{ikx} =\bigl(\sin (Nx/2)\bigr)^2\bigm /\bigl(N (\sin (x/2)\bigr)^2\bigr)
$$
be a univariate Fej\'er kernel.
The Fej\'er kernel ${\mathcal K}_{N}$ is an even nonnegative trigonometric
polynomial in $\Tr(N-1)$.  In the multivariate case define
$$
\mathcal K_{\mathbf N} (\mathbf x) :=\prod_{j=1}^d\mathcal K_{N_j}  (x_j)  ,\qquad
\mathbf N = (N_1 ,\dots,N_d).
$$
Then the $\mathcal K_{\mathbf N}$ are nonnegative trigonometric polynomials from  $\Tr(\mathbf N-\mathbf 1,d)$  which  have
the following property:
\begin{equation}\label{2.2.6}
\|\mathcal K_{\mathbf N}\|_1  = 1.
\end{equation}
Consider the function 
$$
g(\bx):= \sum_{\|\bs\|_1=n} {\mathcal K}_{2^{\bs-2}}(\bx)e^{i(2^\bs-2^{\bs-2},\bx)}.
$$
Then by (\ref{2.2.6}) 
\be\label{4.a}
\|g\|_1 \ll n^{d-1}.
\ee
Take any set $K_m$ of $m$ frequencies. It is clear that for small enough $c(d)$ we have
\be\label{4.b}
\sigma_m(g)_2^2 \ge \sum_{\bk\in \Delta Q_n\setminus K_m} |{\hat g}(\bk)|^2  \gg |\Delta Q_n|.
\ee
Relations (\ref{4.a}) and (\ref{4.b}) imply
$$
\sigma_m(\bW^{a,b}_1)_2 \gg 2^{n(1/2-a)}n^{-(d-1)/2+b(d-1)}\asymp m^{-a+1/2}(\log m)^{(d-1)(b+a-1)}.
$$

\end{proof}

Consider now the case $q=1$, $1<p\le 2$. We need a version of the relation (\ref{1.2}) from Theorem \ref{T1.1} adjusted to our case. 
\begin{Lemma}\label{L4.3} Let $1<p<\infty$. For any $\epsilon >0$ there exists a constant $C(\epsilon,d,p)$ such that for each $t\in \Tr(Q_n)$ we have
$$
\sum_{\|\bs\|_1\le n}\|\delta_\bs(t)\|_p \le C(\epsilon,d,p)n^\epsilon 2^{\bt n} \|t\|_1,\quad \bt=1-1/p.
$$
\end{Lemma}
\begin{proof} Choose $v\in (1,p)$ such that $v':= v/(v-1) >d/\epsilon$. By the H{\" o}lder inequality and Theorem \ref{T1.1} we get
$$
\sum_{\|\bs\|_1\le n}\|\delta_\bs(t)\|_p \le \left(\sum_{\|\bs\|_1\le n}1\right)^{1/v'}\left(\sum_{\|\bs\|_1\le n}\|\delta_\bs(t)\|_p^v\right)^{1/v}
$$
$$
\le n^\epsilon 2^{n(1/v-1/p)} \left(\sum_{\|\bs\|_1\le n}\|\delta_\bs(t)\|_p^v2^{\|\bs\|_1(1/p-1/v)v}\right)^{1/v} \ll n^\epsilon 2^{n(1/v-1/p)} \|t\|_v.
$$
By Theorem \ref{NI} continue
$$
\ll n^\epsilon 2^{n(1/v-1/p)} 2^{n(1-1/v)}\|t\|_1 = n^\epsilon 2^{n\bt} \|t\|_1.
$$
\end{proof}

\begin{Theorem}\label{T4.2} Let $1< p\le 2$. For any $\epsilon>0$ we have
$$
 \sigma_m(\bW^{a,b}_1)_{p}
  \ll  \left\{\begin{array}{ll} m^{-a+\beta}(\log m)^{(d-1)(a+b-2\beta)+\epsilon}, &   a>2\beta,\\
 m^{-a+\beta}(\log m)^{(d-1)b+\epsilon} , &  \beta< a \le 2\beta.\end{array} \right.
$$
\end{Theorem}
\begin{proof} We begin with the case of large smoothness. Let $f\in \bW^{a,b}_1$, $a>2\bt$. 
Then by Lemma \ref{L4.3}
\be\label{4.6}
\sum_{\|\bs\|_1=l}\|\delta_\bs(f_l)\|_p \ll   2^{-(a-\bt)l} l^{b(d-1)+\epsilon}.
\ee
As in the proof of Theorem \ref{Wab} from Section 2 denote by $G_l$ the set of indices $\bs$ with $m_l$ largest $\|\delta_\bs(f_l)\|_p$. Then by Corollary \ref{LP2} and Lemma \ref{Lqp} we obtain for 
$$
f_l':= \sum_{\bs\notin G_l}\delta_\bs(f_l),
$$
$$
\|f_l'\|_p \ll \left(\sum_{\bs\notin G_l}\|\delta_\bs(f_l)\|_p^p\right)^{1/p} 
$$
\be\label{4.7}
\ll (m_l+1)^{-\bt}\sum_{\|\bs\|_1=l} \|\delta_\bs(f_l)\|_p \le (m_l+1)^{-\bt}2^{-(a-\bt)l} l^{b(d-1)+\epsilon}.
\ee
Set for $l>n$
$$
m_l:= [2^{-\kappa(l-n)}n^{d-1}],
$$
where $\kappa>1$ is such that $a>(1+\kappa)\bt$. We define the $m$-term approximant 
$$
A_m(f):= S_{Q_n}(f) +\sum_{l>n} \sum_{\bs\in G_l}\delta_\bs(f_l).
$$
Then 
\be\label{4.8}
m\le |Q_n| +\sum_{l>n}2^l m_l \ll 2^nn^{d-1}.
\ee
For the error of approximation we obtain from (\ref{4.7})
\be\label{4.9}
\|f-A_m(f)\|_p \le \sum_{l>n} \|f_l'\|_p \ll 2^{-(a-\bt)n}n^{(b-\bt)(d-1)+\epsilon}. 
\ee
Relations (\ref{4.8}) and (\ref{4.9}) imply the required upper bound. 

In the case $\bt <a \le 2\bt$ the proof repeats the corresponding argument from the proof of 
Theorem \ref{Wab} in Section 2. Instead of (\ref{2.1}) we use (\ref{4.6}). Also, in the case $a=2\bt$ the factor $\log\log m$ is included in $(\log m)^\epsilon$. 

The lower bounds with $\epsilon =0$ in the case of small smoothness $a\le 2\bt$ follow from the univariate case. We now consider the case of large smoothness $a>2\bt$. In the case $p=2$ it is proved in the proof of Theorem \ref{T4.1}. We use the same example to prove the lower bounds for $p<2$.     Instead of (\ref{4.b}) by Theorem \ref{T1.1} we obtain
\be\label{4.c}
\sigma_m(g)_p^p \gg \sum_{\|\bs\|_1=n}\sum_{\bk\in \rho(\bs)\setminus K_m}\left((|{\hat g})\bk)|^2)^{1/2}2^{n(1/2-1/p)}\right)^p \gg n^{d-1}2^{n(p-1)}.
\ee
Relations (\ref{4.a}) and (\ref{4.c}) imply
$$
\sigma_m(\bW^{a,b}_1)_p \gg 2^{n(1-1/p-a)}n^{(d-1)(b+1/p-1)} \asymp m^{-a+\bt}(\log m)^{(d-1)(a+b-1)}.
$$
\end{proof}

Consider a class ${\bar \bW}^{a,b}_q$, which consists of functions $f$ with a representation
$$
f=\sum_{n=1}^\infty t_n, \quad t_n\in \Tr(Q_n), \quad \|t_n\|_q \le 2^{-an} n^{b(d-1)}.
$$
It is easy to see that in the case $1<q<\infty$ classes ${\bar \bW}^{a,b}_q$ and $\bW^{a,b}_q$ are equivalent. Embedding of $\bW^{a,b}_q$ into ${\bar \bW}^{a,b}_q$ is obvious and the opposite embedding follows from the inequality for $f\in {\bar \bW}^{a,b}_q$
$$
\|f_l\|_q = \|(S_{Q_l}-S_{Q_{l-1}})(f)\|_q \ll \sum_{n\ge l}\|t_n\|_q \ll 2^{-al}({\bar l})^{b(d-1)}.
$$
In the case $q=1$ classes ${\bar \bW}^{a,b}_1$ are wider than $\bW^{a,b}_1$. However, the results of this section hold for these classes as well.
\begin{Remark}\label{R4.1} Theorems \ref{T4.1} and \ref{T4.2} hold for the class ${\bar \bW}^{a,b}_1$ instead of $\bW^{a,b}_1$.
\end{Remark}

\section{Discussion} 

The effect of {\it small smoothness} in the behavior of asymptotic characteristics of smoothness classes was discovered by Kashin \cite{Ka} in 1981. He proved that the rate of decay of the Kolmogorov widths $d_n(W^r_1,L_p)$ of the univariate classes $W^r_1$ depends on $r$ differently in the range $1-1/p<r<1$ (small smoothness) and in the range $r>1$. Belinskii \cite{Be2} studied the univariate $m$-term trigonometric approximation and observed the small smoothness effect in that setting. Romanyuk \cite{Rom1} conducted a detailed study of $m$-term trigonometric approximation of classes of multivariate functions with small mixed smoothness. The Besov classes $\bB^r_{q,\theta}$ are studied in \cite{Rom1}. Define
$$
\|f\|_{\bH^r_q}:= \sup_\bs \|\delta_\bs(f)\|_q 2^{r\|\bs\|_1},
$$
and for $1\le \theta <\infty$ define
$$
\|f\|_{\bB^r_{q,\theta}}:= \left(\sum_{\bs}\left(\|\delta_\bs(f)\|_q 2^{r\|\bs\|_1}\right)^\theta\right)^{1/\theta}.
$$
We write $\bB^r_{q,\infty}:=\bH^r_q$. 
With a little abuse of notation, denote the corresponding unit ball
$$
\bB^r_{q,\theta}:= \{f: \|f\|_{\bB^r_{q,\theta}}\le 1\}.
$$

In case of approximation in $L_p$, $2<p<\infty$, Lemma \ref{BeL} was used in \cite{Rom1}. 
This makes the corresponding results in \cite{Rom1} nonconstructive. We note that the bound for the $m$-term approximation error in Lemma \ref{BeL} follows from Theorem \ref{T2.6} and extra property $\theta_m\subset \theta_n$ in Lemma \ref{BeL} follows from the proof of Theorem \ref{T2.6} in \cite{T150}. Thus, Theorem \ref{T2.6} makes Lemma \ref{BeL} constructive and, therefore, the nonconstructive results from \cite{Rom1}, which are based on Lemma \ref{BeL}, are made constructive in this way. Also, the use of Theorem \ref{T2.6} is technically easier than the use of Lemma \ref{BeL}. For instance, in the proof of upper bounds in Theorem \ref{T3.1} we estimate  $\|g_A\|_A$ in a rather simple way because of additivity property of the norm $\|\cdot\|_A$ and then apply Theorem \ref{T2.6} to $g_A$. 
Typically, in \cite{Rom1} Lemma \ref{BeL} is applied to individual dyadic blocks $\delta_\bs(f)$ with $m_\bs$ depending on the norm of the $\delta_\bs(f)$. It would be interesting to see how much the technique, based on Theorem \ref{T2.6}, could simplify the study of $\sigma_m(\bB^r_{q,\theta})_p$. 

Let us make some comparison of our results on the $\bW^r_q$ classes with known results on $\bB^r_{q,\theta}$ classes. It follows from Corollary \ref{LP2} that for $1<q\le 2$ we have
$$
\|f\|_{\bB^r_{q,2}} \ll \|f\|_{\bW^r_{q}} \ll \|f\|_{\bB^r_{q,q}}.
$$
For example, in the case $\bt p'<r<1/q$ Theorem \ref{T3.1} gives
\be\label{5.1}
\sigma_m(\bW^r_q)_p \asymp m^{-(r-\bt)p/2}(\log m)^{(d-1)(r(p-1)-\bt p)}.
\ee
The corresponding results from \cite{Rom1} give
\be\label{5.2}
\sigma_m(\bB^r_{q,\theta})_p \asymp m^{-(r-\bt)p/2}(\log m)^{(d-1)((r-1/q)(p-1)+1-1/\theta)}.
\ee
In the case $\theta=q$ the right hand sides of (\ref{5.1}) and (\ref{5.2}) coincide. This means that our results for a wider class $\bW^r_q$ imply the corresponding results for a smaller class $\bB^r_{q,q}$. Relations (\ref{5.1}) and (\ref{5.2}) show that $\sigma_m(\bW^r_q)_p$ and $\sigma_m(\bB^r_{q,2})_p$ have different orders. 
 
As we already pointed out in the Introduction the main novelty of the paper is in providing constructive algorithms for optimal $m$-term trigonometric approximation on classes with small mixed smoothness. This is achieved by using Theorem \ref{T2.6}. The use of Theorem \ref{T2.6} is simpler than the use of Lemma \ref{BeL} traditionally used in this area of research. In addition to traditional use of Theorem \ref{T1.1}, which goes back to papers \cite{T29} and \cite{Tmon}, we use other deep results from the hyperbolic cross approximation theory -- Theorem \ref{A}, Theorem \ref{NI} and Lemma \ref{L4.2}. We also prove a new result -- Lemma \ref{L4.3}. These results allowed us to treat the case $q=1$ (see Section 4). 

A number of interesting unresolved problems on $m$-term trigonometric approximation is discussed in \cite{T150}, Section 6. This paper makes a progress in some of them. For instance, Theorems \ref{T3.1I} and \ref{T3.7} cover the case $\bt p' <r\le 1/q$ for constructive $m$-term approximation of $\bW^r_q$ classes. The case $\bt <r\le \bt p'$ is still open. 
There is no progress on small smoothness classes in the case $2\le q<p<\infty$. In the case $q=1$ results presented in Section 4 are optimal up to a factor $(\log m)^\epsilon$ with arbitrarily small $\epsilon>0$. It would be interesting to find right orders of $\sigma_m(\bW^{a,b}_1)_p$ and right orders of constructive $m$-term approximation of these classes. 

The reader can find a detailed discussion of greedy algorithms in Banach spaces in \cite{Tbook} and their applications for the $m$-term trigonometric approximation in \cite{DKTe}, \cite{T12}, \cite{T144}, and \cite{T150}.


\end{document}